\documentclass[12pt,twoside]{article}
\usepackage[colorlinks]{hyperref}
\usepackage[margin=1in]{geometry}
\usepackage{times,amsmath,stackengine,amsthm,amsfonts,amssymb,mathtools,color,latexsym,amsxtra,layout}
\usepackage{graphicx,subfig}
\usepackage{algorithmicx, algpseudocode}
\usepackage{algorithm}
\usepackage[font={small}]{caption}
\usepackage{cleveref}
\usepackage{ulem}
\usepackage{verbatim}
\usepackage{wrapfig,lipsum,booktabs}
\usepackage[mathscr]{euscript}
\usepackage{multirow}
\usepackage{bbm}
\usepackage[titletoc,title]{appendix}
\usepackage{mathabx}
\usepackage{spreadtab}
\usepackage{diagbox}
\usepackage{enumerate}
\newtheorem{thm}{Theorem}

\newtheorem{prop}{Proposition}

\newtheorem{cor}{Corollary}

\newtheorem{remark}{Remark}
\newcommand{\norm}[1]{\left \| {#1}\right \|}
\newcommand{\eqn}[1]{\begin{equation} #1 \end{equation}}
\renewcommand{\epsilon}{\varepsilon}

\newcommand{\eqnnolabel}[1]{\begin{equation*} #1 \end{equation*}}
\newcommand{\R}{\mathbb{R}}

\newcommand{\too}{\longrightarrow}

\author{Hagop Karakazian \thanks{Hagop Karakazian, Department of Mathematics, American University of Beirut, P.O.Box 11-0236, Riad El-Solh, Beirut 1107 2020, Lebanon. (hk93@aub.edu.lb)} \and Toni Sayah \thanks{Toni Sayah, Laboratoire de  Math\'ematiques et Applications, Unit\'ede recherche Math\'ematiques et Mod\'elisation, Facult\'e des Sciences, CAR, Universit\'e Saint-Joseph de Beyrouth, B.P 11-514 Riad El Solh, Beyrouth 1107 2050, Liban. (toni.sayah@usj.edu.lb)} \and Faouzi Triki \thanks{Faouzi Triki, Laboratoire Jean Kuntzmann, UMR CNRS 5224, Universit\'e Grenoble-Alpes, 700 Avenue Centrale, 38401 Saint-Martin-d'H\`eres, France. (faouzi.triki@univ-grenoble-alpes.fr) }}

\title{Recovering the Polytropic Exponent in the Porous Medium Equation: Asymptotic Approach}

\date{}

\begin{document}

\maketitle
\begin{abstract}
\noindent In this paper we consider the time dependent Porous Medium Equation, $u_t = \Delta u^\gamma$ with real polytropic exponent $\gamma>1$, subject to a homogeneous Dirichlet boundary condition. We are interested in recovering $\gamma$ from the knowledge of the solution $u$ at a given large time $T$. Based on an asymptotic inequality satisfied by the solution $u(T)$, we propose a numerical algorithm allowing us to recover $\gamma$. An upper bound for the error between the exact and recovered $\gamma$ is then showed. Finally, numerical investigations are carried out in two dimensions. 
\end{abstract}
\section{Introduction}

Let $\Omega \subset \R^n$, with $n=1,2,3$, be an open domain bounded with boundary $\partial \Omega$. Consider for non-negative initial density $u_0 \in C(\overline{\Omega})$ and real polytropic exponent $\gamma >1$, the following initial-boundary value direct problem for the \textit{Porous Medium Equation}  given by
\eqnnolabel{ \label{PME} \mbox{(PME)}
\begin{cases}
u_t-\Delta u^\gamma =0 & \mbox{ on } \Omega \times \R^+, \\
u(x,0) = u_0(x) & \mbox{ on } \Omega,
\\ u(x,t)=0 & \mbox{ on } \partial \Omega \times \R^+.
\end{cases}}
\noindent This degenerate parabolic equation describes a wide variety of time-dependent phenomena, including gas flow though porous media \cite{muskat37}, plasma diffusion in a magnetic field \cite{longmire63}, and biological population dynamics \cite{gurtin77} (see also Chapter 2 of \cite{vazquez07} and references therein). \\

%
%
\noindent It is well-known (see \cite{vazquez07}) that, under some regularity assumptions on $\Omega$ and $u_0 \geq 0$, the direct (PME) problem has a unique and stable weak solution. Furthermore, it is shown in \cite{aronson81,peletier81} that, this solution is continuous, and in finite time becomes positive and smooth on the entire domain $\Omega$, asymptotically behaving like a separation-of-variables solution to direct (PME) problem. This asymptotic relation serves the basis of our work on recovering $\gamma$ from the knowledge of the solution at some large time $T$. More specifically, we solve the following inverse problem:\\

\noindent \textbf{Inverse (IPME) Problem.} Given a positive measurement $u_T \in L^\infty(\Omega)$ at a large time $T$, recover $\gamma>1$ so that the solution $u(x,t)$ to the direct (PME) problem is such that $u(x,T)=u_T(x)$. \\
%
%

\noindent Many works in literature focus on inverse problems corresponding to Porous Medium type equations. Recently, C\^arstea et. al. \cite{siampaper,Carstea2021AnIB} studied an inverse boundary value problem for uniquely recovering the coefficient functions in front of $u_t$ and $\nabla u^\gamma$ in the Porous Medium Equation with and without an absorption term. We refer the reader to the references in \cite{siampaper} concerning  inverse problems for quasilinear and semilinear elliptic and parabolic equations. \\

\noindent To the best of our knowledge, there is no work done about recovering the polytropic exponent $\gamma$ for the Porous Medium Equation which is the subject on this work. \\

\noindent The outline on the paper is as follows: In section two, we recall some preliminary results on the direct (PME) problem and the asymptotic behaviour of its solution. In section three, we show an asymptotic relation between $u_T$ and $\gamma$. In section four, we propose a numerical algorithm and show the corresponding error between the exact and recovered polytropic exponents. Finally, in section five we present numerical simulations in two dimensions.
%
%
%
\section{Preliminaries}
\label{chpre}
In this section, we introduce some notation and then recall some existence, regularity and asymptotic properties of the solution to the direct (PME) problem.  \\

\noindent In the rest of the paper, we consider the following hypothesis on the domain $\Omega$ and the initial data $u_0$:\\
\noindent ($H_1$)  $\Omega$ is open, bounded, path-connected, has compact closure with $\partial \Omega$ of class $C^3$. \\
\noindent ($H_2$)  $u_0 \geq 0$ is continuous  on $\overline{\Omega}$ with $u_0=0$ on $\partial \Omega$ and $u_0^\gamma \in C^1(\overline{\Omega})$.\\
%
%

\noindent We also recall the following  sets of  functions defined on a time
interval $(0, T)$ with values in a separable functional space $W$:
\[
L^2(0,T; W) = \Big\{ f \mbox{ measurable on }(0,T);\ \displaystyle \int_{0}^{T}\norm{f(t)}_{W}^2 dt<\infty   \Big\}
\]
and
\[
L^2_{loc}(\R^+; W) = \Big\{ f \in L^2((\tau_1,\tau_2); W); \mbox{ for any } 0<\tau_1 < \tau_2 <\infty \Big\}.
\]
\noindent The following theorem gives the well-posedness of the solution to the direct (PME) problem:
%
%
%
%
%
\begin{thm}
\label{thmdirectexist}
\textbf{(Well-Posedness - Theorem 6.12 in \cite{vazquez07})} \\
For any non-negative $u_0 \in L^1(\Omega)$, the direct (PME) problem has a unique and stable weak solution $u \in C([0,\infty),L^1(\Omega))$ with $u^\gamma \in L^2_{loc}(\R^+,H_0^1(\Omega))$ satisfying:
\eqn{\int_0^\infty \int_\Omega \nabla u^\gamma \nabla \eta - u \eta_t dx dt =0}
\noindent for all $\eta \in C_0^1([0,\infty) \times \overline \Omega)$ which vanishes everywhere for $0<t<\tau$ for some $\tau >0$.

\noindent Furthermore when $u_0 \in L^\infty(\Omega)$, the Maximum Principle for parabolic equations implies
\begin{equation}\label{peincmax}
0 \leq u(x,t) \leq \norm{u_0}_{L^\infty(\Omega)}  \quad \mbox{a.e. on } \Omega \times \R^+.
\end{equation}
\end{thm}
\begin{thm}
\label{thmcontinuity}
\textbf{(Continuity - Theorem 3 in \cite{peletier81})} \\
Under the hypothesis ($H_1$), if $u_0 \geq 0$ is continuous on $\overline{\Omega}$ with $u_0=0$ on $\partial \Omega$, then the weak solution $u$ obtained in Theorem \ref{thmdirectexist} is continuous on $\overline{\Omega} \times [0,\infty)$.
\end{thm}
\begin{thm}
\label{thmdomfill}
\textbf{(Domain-Filling Property - Proposition 4 in \cite{aronson81})} \\
Under the hypotheses ($H_1$) and ($H_2$) with $u_0 \not \equiv 0$, there exists a finite domain-filling time $T_{fill} \geq 0$ depending only on $u_0$, such that the weak solution obtained in Theorem \ref{thmdirectexist} is stictly positive on $\Omega \times [T_{fill},\infty)$.
\end{thm}
%
%
%
%
%
%
%
%
%

\noindent 

\noindent Observe that on $\Omega \times [T_{fill},\infty)$, (PME) becomes a uniformly (non-degenerate) parabolic equation with smooth coefficients, and so $u$ is smooth there. \\

\noindent Apart from weak solutions, there is a family (for $\tau > 0$) of smooth separation-of-variables solutions to \eqref{PME} given by
\eqn{\label{friendlygiant} U_\tau(x,t):= (\tau+t)^{-\frac{1}{\gamma-1}} f(x)}
corresponding to initial data $u_0=\tau^{-\frac{1}{\gamma-1}} f$, where $f\in C^\infty(\Omega)$ is the unique and strictly positive solution of the non-linear eigenvalue problem
\eqn{\label{eigen} \begin{cases} -\Delta f^\gamma = \frac{1}{\gamma-1} f & \mbox{ on } \Omega,
\\ f=0 & \mbox{ on } \partial \Omega,\end{cases}}
\noindent obtained in \cite{aronson81}.\\
%

%

%
\begin{thm}
\label{uTzero}
\textbf{(Universal Decay - Theorems 1 and 2 in \cite{aronson81})} \\
Under the hypotheses ($H_1$) and ($H_2$) with $u_0 \not \equiv 0$, the weak solution $u$ obtained in Theorem \ref{thmdirectexist} satisfies for large time $t$ the following decay property: 
\eqn{\label{decay} \frac{f(x)}{(\tau_0+t)^{\frac{1}{\gamma-1}}} \le u(x,t) \leq \frac{f(x)}{(\tau_1+t)^{\frac{1}{\gamma-1}}}, \quad \mbox{ on } \Omega.}
where $\tau_0$ and $\tau_1$ are positive constants depending on $u_0$ and $\Omega$ (with $\tau_0 > \tau_1$).
\end{thm}
%
%

\noindent Theorem \ref{uTzero} implies the asymptotic form of $u$ as follows.

\begin{thm} \label{thmasympme} \textbf{(Asymptotic Behavior - Main Result of \cite{aronson81})} \\
Under the hypotheses ($H_1$) and ($H_2$), the weak solution $u$ obtained in Theorem \ref{thmdirectexist} satisfies for large time $t$:
\eqn{\label{eqnasympme}  \norm{u(x,t)-(1+t)^{\frac{-1}{\gamma-1}}f(x)}_{L^\infty(\Omega)}\leq \frac{K_\gamma}{(1+t)^{1+\frac{1}{\gamma-1}}} }
\noindent where $K_\gamma>0$ is a constant depending only on $\gamma, u_0$ and $\Omega$. 
\end{thm} 
%
%

%
\section{Inverse Problem: The Asymptotic Method}
\label{asymptotic}

In the rest of the paper, we assume the hypotheses ($H_1$) and ($H_2$)
are satisfied. \\
\noindent We begin with a proposition that gives an asymptotic relation between $u_T$ and $\gamma$. For this purpose, we introduce the function
\begin{equation}\label{compw}
w:=(-\Delta)^{-1} u_T \in H_0^1(\Omega),
\end{equation}
which is the unique solution to Poisson's equation with homogeneous Dirichlet boundary condition.
\begin{prop} \label{keyprop}
Under the hypotheses ($H_1$) and ($H_2$) in dimensions $n=1,2,3$, we have for large time $T$:
\eqn{\label{asymput} \norm{(\gamma-1)u_T^\gamma-\frac{w}{1+T}}_{L^\infty(\Omega)} \leq \frac{C_\gamma}{(1+T)^{2+\frac{1}{\gamma-1}}} 
}
where 
$C_\gamma>0$ is a constant depending only on $\gamma, u_0$ and $\Omega$.
\end{prop}

\begin{proof}
Setting $t=T$ in  \eqref{eqnasympme}, we obtain
\eqn{\label{e1} (1+T)^{-\frac{1}{\gamma-1}} f = u_T + h \quad \mbox{ a.e. on } \Omega}
with error term $h \in L^\infty(\Omega)$ satisfying 
\begin{equation}\label{hhh}
\begin{array}{l}
\norm{h}_{L^\infty(\Omega)}\leq \displaystyle\frac{K_\gamma}{(1+T)^{1+\frac{1}{\gamma-1}}}.
\end{array}
\end{equation}
Raising both sides of \eqref{e1} to the power $\gamma$, and using Taylor's Theorem on the function $j(\cdot):=(u_T+\cdot)^\gamma$ (which is $C^1$ around 0 because $\gamma>1$ and $u_T > 0$ as $T\geq T_{fill}$), we obtain
\eqn{\label{e2} (1+T)^{-\frac{\gamma}{\gamma-1}} f^\gamma= (u_T + h)^\gamma = u_T^\gamma + \gamma(u_T+c)^{\gamma-1}h \quad \mbox{ a.e. on } \Omega}
for some function $c \in L^\infty(\Omega)$  such that 
\eqn{\label{e2.5} \norm{c}_{L^\infty(\Omega)} \leq \norm{h}_{L^\infty(\Omega)}.}
\noindent On the other hand, $f$ satisfies the following variational formulation of the non-linear eigenproblem \eqref{eigen}:
$$\int_\Omega \nabla[(\gamma-1)f^\gamma] \cdot \nabla v = \int_\Omega f v \quad \forall v \in H_0^1(\Omega).$$
And so through the uniqueness of the Poisson's equation with zero Dirichlet boundary conditions, we have
\eqn{\label{e3} (\gamma-1)f^\gamma = (-\Delta)^{-1} f \quad \mbox{ a.e. on } \Omega,}
or equivalently when multiplying both sides by $(1+T)^{-\frac{1}{\gamma-1}}$,
\eqn{\label{e4} (\gamma-1)(1+T)(1+T)^{-\frac{\gamma}{\gamma-1}}f^\gamma = (-\Delta)^{-1} [(1+T)^{-\frac{1}{\gamma-1}}f] \quad \mbox{ a.e. on } \Omega.}
Substituting equations \eqref{e1} and \eqref{e2} into \eqref{e4}, we get rid of $f$ and obtain
\eqn{\label{e5} (\gamma-1)(1+T)[u_T^\gamma +\gamma(u_T+c)^{\gamma-1}h] = (-\Delta)^{-1} [u_T + h] \quad \mbox{ a.e. on } \Omega.}
Expanding, dividing by $(1+T)$,  rearranging and using \eqref{compw}, we obtain
\eqn{\label{e6} (\gamma-1)u_T^\gamma -\frac{w}{1+T} = \frac{1}{1+T}(-\Delta)^{-1} h-\gamma(\gamma-1)(u_T+c)^{\gamma-1}h \quad \mbox{ a.e. on } \Omega.}
Now via triangle inequality, and the relations  \eqref{e2.5},   \eqref{hhh} and \eqref{decay}, we obtain
\begin{equation}\label{extention1}
\begin{array}{rcl}
\medskip
\displaystyle \norm{(\gamma-1)u_T^\gamma-\frac{w}{1+T}}_{L^\infty(\Omega)} &\leq& \displaystyle \frac{1}{1+T}\norm{(-\Delta)^{-1} h}_{L^\infty(\Omega)}\\
\medskip
&& +\gamma(\gamma-1) (\norm{u_T}_{L^\infty(\Omega)}+\norm{h}_{L^\infty(\Omega)})^{\gamma-1} \norm{h}_{L^\infty(\Omega)}\\
\medskip
&\leq& \displaystyle \frac{1}{1+T}\norm{(-\Delta)^{-1} h}_{L^\infty(\Omega)}\\
&& \displaystyle + \displaystyle \gamma(\gamma-1) \left [\frac{\norm{f}_{L^\infty(\Omega)}}{(\tau_1+T)^{\frac{1}{\gamma-1}}} + \displaystyle\frac{K_\gamma}{(1+T)^{1+\frac{1}{\gamma-1}}}\right ]^{\gamma-1} \norm{h}_{L^\infty(\Omega)}.
\end{array}
\end{equation}

\noindent Setting $\tau_2=\min(\tau_1,1)$, this implies

\begin{equation}\label{extention}
\begin{array}{rcl}
\medskip
\displaystyle \norm{(\gamma-1)u_T^\gamma-\frac{w}{1+T}}_{L^\infty(\Omega)} &\leq& \displaystyle \frac{1}{\tau_2+T}\norm{(-\Delta)^{-1} h}_{L^\infty(\Omega)}\\
\medskip
&& \displaystyle + \displaystyle \gamma(\gamma-1) \left [\norm{f}_{L^\infty(\Omega)}+ \displaystyle\frac{K_\gamma}{\tau_2+T}\right ]^{\gamma-1} \frac{\norm{h}_{L^\infty(\Omega)}}{\tau_2+T},
\end{array}
\end{equation}
\noindent where 
$$\norm{(-\Delta)^{-1} h}_{L^\infty(\Omega)} \leq C_\infty \norm{(-\Delta)^{-1} h}_{H^2(\Omega)} \leq C_\infty C_E \norm{h}_{L^2(\Omega)} \leq C_\infty C_E |\Omega|^{1/2} \norm{h}_{L^\infty(\Omega)}$$
with $C_\infty>0$ is the constant from the continuous embedding $H^2(\Omega)  \subset L^\infty(\Omega)$
when $n<4$, and $C_E>0$ is the elliptic regularity constant. Putting all together, along with the fact that $\frac{K_\gamma}{\tau_2+T} \leq \frac{K_\gamma}{\tau_2}$, we obtain
\begin{equation}\label{almost} 
\begin{array}{rcl}
\medskip
\displaystyle \norm{(\gamma-1)u_T^\gamma-\frac{w}{1+T}}_{L^\infty(\Omega)} &\leq& \Big[ C_\infty C_E |\Omega|^{1/2}+\gamma(\gamma-1)\big(\norm{f}_{L^\infty(\Omega)}+ \displaystyle K_\gamma/\tau_2\big) \Big] \displaystyle \frac{\norm{h}_{L^\infty(\Omega)}}{\tau_2+T}.\\
%
\end{array}
\end{equation}
The last factor in inequality \eqref{almost} can be treated by \eqref{hhh} as follows: ($\tau_2 \le 1$)
\begin{equation}\label{asynptot11}
\displaystyle  \frac{\norm{h}_{L^\infty(\Omega)}}{(\tau_2+T)} \le  \displaystyle\frac{K_\gamma}{(1+T)^{2+\frac{1}{\gamma-1}}} \Big( \frac{1+T}{\tau_2+T}\Big) \le \displaystyle\frac{K_\gamma}{(1+T)^{2+\frac{1}{\gamma-1}}} \Big(\frac{1-\tau_2}{\tau_2}+1 \Big ).
\end{equation}
\noindent From which the desired result follows.
%
\end{proof}
\begin{cor}\label{corr1}
Under the hypothesis of Proposition \ref{keyprop}, let $\gamma$ be a solution to the inverse (IPME) problem (which satisfies \eqref{asymput} for a large time $T$). Then for any real $\alpha >1$ satisfying the following relation:
\eqn{\label{alphaasymp} \norm{(\alpha-1)u_T^\alpha-\frac{w}{1+T}}_{L^\infty(\Omega)} \leq \frac{C_\gamma}{(1+T)^{2+\frac{1}{\gamma-1}}} 
,}
where $C_\gamma$ is given in Proposition \ref{keyprop}, we have the following bound:
\begin{equation}\label{estminv}
\begin{array}{l}
\medskip
\displaystyle |\gamma - \alpha | \le \displaystyle 
 \frac{\hat{C}_\gamma}{1+T}
\end{array},
\end{equation}
where $\hat{C}_\gamma$ is a positive constant depending on $C$, $\gamma$, $u_0$ and $\Omega$.
\end{cor}
\begin{proof}
Let $\gamma$ be a solution to the inverse (IPME) problem satisfying \eqref{asymput} for a large time $T$. Without loss of generality, assume $\alpha \in (1,\gamma]$ is a real number satisfying \eqref{alphaasymp}. 
Then by Proposition \ref{keyprop} and triangle inequality, we obtain for a large time $T$:
\begin{equation}\label{interm}
\begin{array}{rcl}
\norm{(\gamma-1)u_T^\gamma - (\alpha-1)u_T^\alpha}_{L^\infty(\Omega)} &\leq& \displaystyle \norm{(\gamma-1)u_T^\gamma-\frac{w}{1+T}}_{L^\infty(\Omega)}+\norm{(\alpha-1)u_T^\alpha-\frac{w}{1+T}}_{L^\infty(\Omega)} \\
&\leq& \displaystyle \frac{2 C_\gamma }{(1+T)^{2+\frac{1}{\gamma-1}}}.
\end{array}
\end{equation}

\noindent Observe that the function $f_\epsilon (x)=(x-1)\epsilon^x$ is differentiable and strictly decreasing on $(\eta_\epsilon,\infty)$ where $\eta_\epsilon= \displaystyle 1-\frac{1}{\ln \epsilon}$, and so injective there. 
As $f=0$ on $\partial \Omega$,  we can pick an arbitrary $x_0 \in \Omega$ near the boundary $\partial \Omega$ such that for all $x\in \Omega$ with $f(x_0) \le f(x)$, we have for large time $T$ and by using Theorem \ref{uTzero},
\begin{equation}\label{ineg1}
\uline{\epsilon} :=  \frac{f(x_0)}{(\tau_0+T)^{\frac{1}{\gamma-1}}} \le   \frac{f(x)}{(\tau_0+T)^{\frac{1}{\gamma-1}}} \le  u_T(x)
\end{equation}
for a positive  small $\uline \epsilon$  ($0< \uline{\epsilon}  < 1$). We consider $\epsilon=u_T(x)$, $x\in \Omega$ satisfying \eqref{ineg1}, so that $f_\epsilon$ will be injective on $(\eta_{\epsilon},\infty)$ where we can define the corresponding continuous inverse function $f_\epsilon^{-1}$ such that: 
\begin{equation}\label{fp}
(f_\epsilon^{-1})'(y)= \displaystyle \frac{f_\epsilon^{-1}(y) -1}{y(1+(f_\epsilon^{-1}(y) -1)\ln(\epsilon))}, \qquad \forall y\in ]0,f_\epsilon(\eta_\epsilon)[. \end{equation}  
%

\noindent Now by the Mean Value theorem applied to $f_\epsilon^{-1}$ over the interval 
$$I_\epsilon=[f_\epsilon(\gamma), f_\epsilon(\alpha)],$$
there exists a real $\xi \in I_\epsilon$ such that,
\begin{equation}\label{equat11}
\begin{array}{rcl}
\medskip
\displaystyle |\gamma - \alpha | &\le& |(f_\epsilon^{-1})'(\xi)|  \, |(\gamma-1)\epsilon^\gamma - (\alpha-1)\epsilon^\alpha |.
\end{array}
\end{equation}
Note that $f_\epsilon^{-1}(\xi) \in [\alpha,\gamma]$ as $\xi \in I_\epsilon$. Now choose a sufficiently small $\epsilon>0$ so that $\alpha >  1- \displaystyle \frac{2}{\ln(\epsilon)} > \eta_\epsilon$, then we get 
\[
\begin{array}{rcl}
\displaystyle |1+ (f_\epsilon^{-1}(\xi) -1)\ln(\epsilon)| &>& \displaystyle |1+(\alpha-1)\ln(\epsilon)|\\
&>& 1.
\end{array}
\]
Using the above inequality, the fact that $\xi \in I_\epsilon$, Theorem \ref{uTzero} and Relation \eqref{fp} gives the following:
\begin{equation}\label{equat22}
\begin{array}{rcl}
\medskip
|(f_\epsilon^{-1})'(\xi)| &<&   \displaystyle  \frac{\gamma -1}{f_\epsilon(\gamma)} \\
&<& \epsilon^{-\gamma} \\
&<& \uline\epsilon^{-\gamma} = \displaystyle \frac{(\tau_0+T)^{\frac{\gamma}{\gamma-1}} }{f^\gamma(x_0)}
\end{array}
\end{equation}
%
%
Then we get by using the relations \eqref{equat11}, \eqref{equat22} and \eqref{interm} ($\tau_0 > \tau_1 >\tau_2$),
\[
\begin{array}{rcl}
\medskip
\displaystyle |\gamma - \alpha | &\le& \displaystyle  \frac{\tilde{C}_\gamma}{(1+T)}\Big( \frac{ \tau_0 +T}{1+T} \Big)^{\frac{\gamma}{\gamma-1}},
\end{array}
\]
where $\tilde{C}_\gamma = \displaystyle \frac{2 C_\gamma}{f^\gamma(x_0)}$. Let $\tau_3=\min(\tau_0,1)$. Then we get:
\begin{equation}\label{bound11}
\begin{array}{rcl}
\medskip
\displaystyle |\gamma - \alpha | &\le& \displaystyle 
\frac{\tilde{C}_\gamma}{1+T} \Big( \frac{ \tau_0+T}{\tau_3 + T} \Big)^{\frac{\gamma}{\gamma-1}} \le \displaystyle 
\frac{\tilde{C}_\gamma}{1+T} \Big( 1+ \frac{ \tau_0-\tau_3}{\tau_3} \Big)^{\frac{\gamma}{\gamma-1}} .
\end{array}
\end{equation}
Finally repeating the above argument for $\alpha > \gamma $, with the only difference at the level of inequality \eqref{equat22} that $\epsilon^{-\alpha} < \epsilon^{-\gamma}$, we reach to
\[
\begin{array}{l}
\medskip
\displaystyle |\gamma - \alpha | \le \displaystyle 
 \frac{\hat{C}_\gamma}{1+T},
\end{array}
\]
completing the proof.

\end{proof}
%
%
%
%
\begin{remark}
The estimate \eqref{estminv} is very important in the sense that:  for a large $T$, any computed  $\alpha$ satisfying \eqref{alphaasymp} will be very close to the exact $\gamma$ of the direct (PME) problem since the latter also satisfies \eqref{alphaasymp}. 
\end{remark}
\noindent Relation \eqref{asymput} gives an asymptotic estimate satisfied by $u_T$ with $L^\infty(\Omega)$-norm. It is straightforward to obtain similar bound with the $L^p(\Omega)$-norm for all $p\ge 1$.
\begin{prop}\label{assymptL1}
Under the hypotheses ($H_1$) and ($H_2$) in dimensions $n=1,2,3$, we have for large time $T$: $\forall p\ge 1$,
\begin{equation}\label{assL1} 
\norm{(\gamma-1)u_T^\gamma-\frac{w}{1+T}}_{L^p(\Omega)} \leq \frac{C_\gamma |\Omega|^{1/p} }{(1 +T)^{2+\frac{1}{\gamma-1}}},
\end{equation}
where $C_\gamma$ is the constant appearing in Proposition \ref{keyprop}, which depends only on $\gamma, u_0$ and $\Omega$.
\end{prop}
\noindent As the numerical algorithm introduced in the next section will be based on the $L^1(\Omega)$-norm, we will establish a corollary analogue to Corollary \ref{corr1} with the $L^p(\Omega)$-norm for $p\ge 1$.
\begin{cor}\label{estimat} Under the hypothesis of Proposition \ref{assymptL1}, let $p \ge 1$ and $\gamma$ be a solution to the inverse (IPME) problem (which satisfies \eqref{assL1} for a large time $T$). Then for any real $\alpha >1$ satisfying the following relation:
\begin{equation}\label{assL11} 
\norm{(\alpha-1)u_T^\alpha-\frac{w}{1+T}}_{L^p(\Omega)} \leq \frac{C_\gamma |\Omega|^{1/p}}{(1+T)^{2+\frac{1}{\gamma-1}}},
\end{equation}
we have the following bound:
\begin{equation}\label{estminv1}
\begin{array}{l}
\medskip
\displaystyle |\gamma - \alpha | \le \displaystyle 
 \frac{\hat C_p}{1+T}
\end{array},
\end{equation}    
where $\hat C_p$ is a positive constant depending on $p$, $C$, $\gamma$, $u_0$ and $\Omega$ . 
\end{cor}
\begin{proof}
Let $\gamma$  be a solution to the inverse (IPME) problem satisfying \eqref{assL1} for a large time $T$ and let $\alpha>1$ satisfying \eqref{assL11}. Then by the triangle inequality, we obtain the following bound: $\forall p\ge 1$,
\begin{equation}\label{equat33}
\norm{(\gamma-1)u_T^\gamma-(\alpha-1)u_T^\alpha}_{L^p(\Omega)} \leq \frac{2 C_\gamma  |\Omega|^{1/p} }{(1+T)^{2+\frac{1}{\gamma-1}}}.
\end{equation}
Then as in the proof of Corollary \ref{corr1}, we consider the same function $f_\epsilon=(x-1)\epsilon^x$ which is differentiable and strictly decreasing on $(\eta_\epsilon,\infty)$, where $\eta_\epsilon=1-\frac{1}{\ln \epsilon}$. Furthermore, as $f=0$ on $\partial \Omega$, 
we can pick two points $x_0$ and $x_1$ in $\Omega$ near $\partial \Omega$ such that $f(x_0) < f(x_1)$, and then extract  a compact subset $\mathcal{O}_c =f^{-1}([f(x_0),f(x_1)]) \subset \Omega$ 
such that $|\mathcal{O}_c|\ne 0$ and  for large time $T$, 
$$ \uline{\epsilon} :=  \frac{f(x_0)}{(\tau_0+T)^{\frac{1}{\gamma-1}}}  \le u_T(x) \le   \overline{\epsilon} := \frac{f(x_1)}{(\tau_1+T)^{\frac{1}{\gamma-1}}}, \quad \forall x \in \mathcal{O}_c$$  
for arbitrary small $0< \uline{\epsilon} < \overline{\epsilon} < 1$. Then, we consider an $x\in \mathcal{O}_c$, denote $\uline{\epsilon} \le \epsilon:=u_T(x)\le \overline{\epsilon}$, so that $f_\epsilon$ will be injective on the interval $(\eta_{\epsilon},\infty)$. One can follow the same steps of the proof of Corollary \ref{corr1} to get the following bound:
\begin{equation}\label{equat111}
\begin{array}{rcl}
\medskip
\displaystyle |\gamma - \alpha | &\le&
\displaystyle \frac{(\tau_0+T)^{\frac{\gamma}{\gamma-1}} }{f^\gamma(x_0)} \, |(\gamma-1)\epsilon^\gamma - (\alpha-1)\epsilon^\alpha |\\
&\le& \displaystyle \frac{(\tau_0+T)^{\frac{\gamma}{\gamma-1}} }{f^\gamma(x_0)}  \, |(\gamma-1)(u_T(x))^\gamma - (\alpha-1)(u_T(x))^\alpha |.
\end{array}
\end{equation}
Finally, by raising the last inequality to the power $p$, integrating over $\mathcal{O}_c \subset \Omega$, using \eqref{equat33} and then raising again to the power $1/p$, we get by denoting $\tau_3=\min(1,\tau_0)$,
\[
\begin{array}{rcl}
\medskip
|\gamma - \alpha | &\le& \displaystyle 
 \frac{\hat{C}_\gamma}{1+T} \Big(\frac{ \tau_0+T}{1+T} \Big)^{\frac{\gamma}{\gamma-1}} \le \displaystyle 
 \frac{\hat{C}_\gamma}{1+T} \Big(\frac{ \tau_0+T}{\tau_3+T} \Big)^{\frac{\gamma}{\gamma-1}}
\\
&\le& \displaystyle 
 \frac{\hat{C}_\gamma}{1+T} \Big( 1+   \frac{ \tau_0-\tau_3}{\tau_3} \Big)^{\frac{\gamma}{\gamma-1}}.
 \end{array}
\]
Thus we get the desired result.
\end{proof}
%
%
%
\noindent In the sequel, we propose an algorithm for approximating a solution $\gamma$ of the inverse (IPME) problem based on \eqref{assL1} for $p=1$. 

\section{A Numerical Algorithm}
In this section, we will propose an algorithm for approximating the solution of the inverse problem for large time $T$, without the knowledge of the initial data $u_0$. For this purpose, we  introduce the function $F:(1,+\infty) \too \R$  given by:
\begin{equation}\label{FF}
F(\alpha)= (\alpha-1)(1+T) u_T^\alpha - w,
\end{equation}
\noindent through which relation \eqref{assL1} can be rewritten as: 
%
%
%
\begin{equation}\label{Pr}
\norm{F(\gamma)}_{L^1(\Omega)} 
\le  \displaystyle \frac{\hat{C}}{(1+T)^{1+\frac{1}{\gamma-1}}}. 
\end{equation}
%
%
%
Based on \eqref{Pr}, our algorithm searches for $\gamma_m$ such that
\begin{equation}\label{minig}
\norm{F(\gamma_m)}_{L^1(\Omega)}  = \displaystyle \underset{\alpha >1}{\min} \norm{F(\alpha)}_{L^1(\Omega)} .
\end{equation}
\begin{remark}
Under the assumption of Theorem \ref{thmcontinuity}, the function $F$ is continuous with respect to $\alpha > 1$ and  $x\in \overline{\Omega}$. Furthermore as $u_T <1$ for large time $T$, the first term $(\alpha-1)(1+T) u_T^\alpha$  of $F(\alpha)$ tends to $0$ when $\alpha$ tends to $+\infty$. This means $|F(\alpha)|$ increases to $w$ as $\alpha$ tends to $+\infty$. Thus there exists $\gamma_c>1$ such that $F(\alpha) > F(\gamma_c)$ for all $\alpha> \gamma_c$. Also, as $\alpha$ tends to $1$, $|F(\alpha)|$ approaches to $w$. Therefore, the minimization problem \eqref{minig} admits at least one solution $\gamma_m \in (1,\gamma_c)$.
\end{remark}

\noindent Following this remark, the above minimization problem becomes:
\begin{equation}\label{mini}
\norm{F(\gamma_m)}_{L^1(\Omega)}  = \displaystyle \underset{1< \alpha < \gamma_c}{\min} \norm{F(\alpha)}_{L^1(\Omega)} 
\end{equation}
\begin{prop}
Under the hypotheses ($H_1$) and ($H_2$) in dimensions $n=1,2,3$, and for large time $T$, any solution $\gamma$ of the inverse (IPME) problem and any solution $\gamma_m$ of the minimization problem \eqref{mini} satisfy the  following estimate:
\begin{equation}\label{gagm}
\displaystyle |\gamma - \gamma_m | \le \displaystyle 
 \frac{\hat C_1}{1+T}
\end{equation}
where $\hat C_1$ is a positive constant depending on $\gamma_c$, $\gamma$, $u_0$ and $\Omega$.
\end{prop}
\begin{proof}
Let $\gamma$ be a solution of the inverse (IPME) problem (which satisfies \eqref{assL1} for large time $T$) and $\gamma_m$ a solution of \eqref{mini}. Then
\begin{equation}\label{fgm}
\norm{(\gamma_m-1)u_T^{\gamma_m}-\frac{w}{1+T}}_{L^1(\Omega)} \leq \norm{(\gamma-1)u_T^{\gamma}-\frac{w}{1+T}}_{L^1(\Omega)} \leq \frac{C_\gamma |\Omega|}{(1+T)^{2+\frac{1}{\gamma-1}}} 
.
\end{equation}
Now the result follows from Corollary \ref{estimat}, where the dependence of $\hat C_1$ on $\gamma_m$ could be removed and replaced with that on $\gamma_c$, as $1 < \gamma_m < \gamma_c$.
\end{proof}
\begin{remark} One can recover $\gamma$ as a solution $\gamma_m$ of the minimization problem \eqref{mini} which satisfies the error estimate \eqref{gagm}.
\end{remark}

\noindent Let us now propose the corresponding numerical algorithm.\\

\noindent \textbf{Algorithm IP :} Given $u_T$ at a large time $T$,  
\begin{enumerate}
\item  We compute $w$ given by \eqref{compw} solution of the following weak formulation: search $w\in H^1_0(\Omega)$ such that,
\begin{equation}\label{weakw}
\displaystyle \int_\Omega \nabla w \nabla v = \displaystyle \int u_T v, \qquad \forall v \in H^1_0(\Omega).
\end{equation}
\item We search for $\gamma_m$ a solution of the minimization problem \eqref{mini} which approximates a solution  $\gamma$ of the inverse (IPME) problem.\\
\end{enumerate}
\section{Numerical Simulations}
In this section, we verify the results obtained in the previous sections numerically in 2D, with computations carried out using Freefem++ software (see \cite{hecht}). More specifically, for different values of $\gamma$, $T$ and $u_0$, we first compute $u_T$ by linearizing and solving the direct (PME) problem with the Lagrange $P_1$  finite elements method for the space discretization and  semi-implicit Euler method for the time discretization. We then approximate, for each case, a solution $\gamma_m$ of the inverse problem through ``Algorithm IP''. We finally analyze the behaviour of the error $|\gamma - \gamma_m|$ with respect to various magnitudes of $T$ and $u_0$ used. \\

\noindent We consider the square $\Omega=(0,1)^2$ in 2D, on which we construct a regular mesh by dividing each edge into $N$ segments. So the corresponding meshe contains $2N^2$ triangles. \\

\noindent In each execution of the ``Algorithm IP'', we first compute $w$ from $u_T$, by solving the system \eqref{weakw} using the classical $P_1$  Lagrange finite elements with space step $h=1./N$. We then consider a sufficiently large real $\gamma_c$ and look for $\gamma_m \in (1, \gamma_c)$ a solution of \eqref{mini}. 
%
\subsection{2D Simulations}
%
%
We begin this section by generating $u_T$ by solving the direct (PME) problem for $\gamma=3.5$,  $T=1000$, and initial data 
$$u_0(x,y)=10\, xy(1-x)(1-y),$$
using $N=10$ (the mesh step is then $h=1/N=0.1$) and considering the time step $dt=h$. We then  calculate $\norm{F(\alpha)}_{L^\infty(\Omega)} $ for $\alpha \in [1.001, 10]$. The left graph in Figure \ref{curvefg2d}  shows the function $F$ with respect to $\alpha$ while the right graph shows a corresponding zoom on the flat part for $\alpha\in [2,10]$. We can clearly see that it has a minimum $\gamma_m=3.504$ which is close to the exact $\gamma=3.5$. The shape of the curve of $F$ is expected if we look at its expression $F(\alpha)=(\alpha -1)(T+1) u_T^\alpha - w$. The first part of $F$ given as $(\alpha -1)(T+1) u_T^\alpha$ tends to zero when $\alpha $ becomes very big as $u_T(x) < 1$ for large $T$, and consequently the corresponding curve becomes flat near to $0$ for big $\alpha$. Consequently, the curve of $F$ becomes flat equal to $|w|$ for big $\alpha$ and this fact explain the existence of a minimum $\gamma_m$. \\
\noindent Similar curves are obtained when we performed similar simulation with other values of the exact $\gamma$. \\

\begin{figure}[!h]
\begin{center}
\vspace{-.3cm}
\includegraphics[width=8cm]{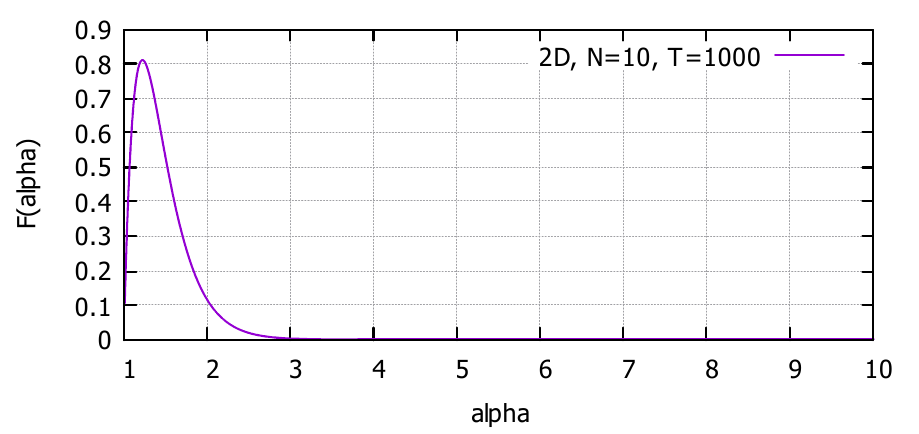} \hspace{.0cm} 
\includegraphics[width=8cm]{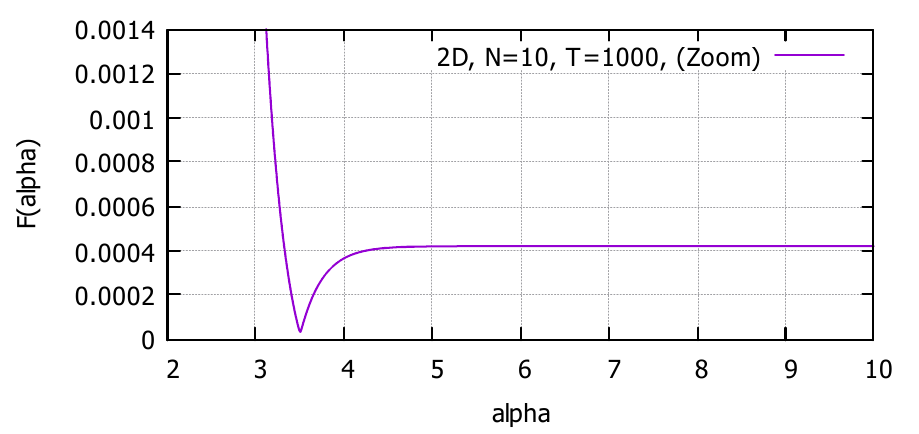}
\vspace{-.2cm}
\caption{Right:  The function $F$ with respect to $\alpha$ for $N=10$, $T=1000$ and $\gamma=3.5$  in $2D$. Left: a  zoom.}\label{curvefg2d}
\vspace{-.5cm}
\end{center}
\end{figure}

\noindent Table \ref{Tab1} shows the  computed $\gamma_m$ with respect to $T$ by using "Algorithm IP".
We deduce that the computed $\gamma_m$ solution of the inverse problem converges to the exact $\gamma$ with respect to the time $T$ which is in coherence with the theoretical results \eqref{gagm}. We note also that the convergence is faster for small values of $\gamma$ and this fact can be interpreted by the fact that the constant appearing in the right hand side of Inequality \eqref{almost} depends on $\gamma$. \\
%

%
%
\begin{table}[h!]
\center
\begin{tabular}{|l|l|l|l|l|l|l|l|l|l|l|l|}
\hline
Algorithm IP ($\gamma_m$) &  T=1 & T=5 & T=10 & T=50  & T=100 &  T=500 & T=$10^3$ & T=$10^4$  \\
 \hline 
($\gamma = 1.1$) & 1.269 & 1.121 & 1.109 & 1.102 & 1.101   & 1.1 & 1.1 & 1.1   \\
\hline
($\gamma = 3.5$) & 3.876 & 3.575  & 3.537 & 3.509 & 3.506 &  3.505 & 3.504 & 3.503   \\
\hline
($\gamma = 6.8$) & 6.592 & 6.653  & 6.711 & 6.784 & 6.795  & 6.804 & 6.803 & 6.802  \\
\hline
($\gamma = 10.2$) & 8.810 & 9.117  & 9.319 & 9.765 &  9.922  & 10.125 & 10.161 & 10.202   \\
\hline
\end{tabular}
\caption{The computed $\gamma_m$ in 2D with respect $T$ for $\gamma = 1.1,3.5,6.8,10.2$ by using "Algorithm IP".}
\label{Tab1}
\end{table}

\end{document}